\documentclass[utf8,10pt]{article}
\usepackage[utf8]{inputenc}
\usepackage{fullpage}
\usepackage{authblk}

\usepackage{amsthm,amsmath,amssymb}
\usepackage{algorithm2e}
\newcommand{\qedclaim}{\hfill $\diamond$ \medskip}

\usepackage[colorlinks=true,citecolor=black,linkcolor=black,urlcolor=blue]{hyperref}

\theoremstyle{plain}
\newtheorem{theorem}{Theorem}
\newtheorem{lemma}[theorem]{Lemma}
\newtheorem{corollary}[theorem]{Corollary}
\newtheorem{proposition}[theorem]{Proposition}

\newtheorem{claim}{Claim}[theorem]

\theoremstyle{definition}

\newtheorem{conjecture}[theorem]{Conjecture}

\newtheorem{question}[theorem]{Question}

\theoremstyle{remark}

\usepackage{thm-restate}

\newenvironment{subproof}{\par\noindent {\it Subproof}.\ }{\hfill$\lozenge$\par\vspace{11pt}}

\newcommand{\cste}{ 8k \cdot (2\cdot \col + 14k) \cdot (2 \cdot \dr)}

\newcommand{\dr}{ (2k+1)\cdot k \cdot (4k)^{4k}}
\newcommand{\col}{(6k^2)^{3k}}
\newcommand{\free}{subdivision-free}

\DeclareMathOperator{\Dis}{Dis}
\DeclareMathOperator{\Len}{Len}

\title{Bispindles in strongly connected digraphs with large chromatic number}

\author[1]{Nathann Cohen}
\author[2,3]{Fr\'ed\'eric Havet}
\author[2,3,4]{William Lochet}
\author[5]{Raul Lopes}
\affil[1]{ CNRS, LRI, Univ. Paris Sud, Orsay, France}
\affil[2]{ Universit\'e C\^ote d'Azur, CNRS, I3S, UMR 7271, Sophia Antipolis, France}
\affil[3]{ INRIA, France}
\affil[4]{ LIP, ENS de Lyon, France}
\affil[5]{ Departamento de Computa\c{c}ao, Universidade Federal do Cear\'a, Fortaleza, Brazil}

\begin{document}

\maketitle

\begin{abstract}
A {\it $(k_1+k_2)$-bispindle} is the union of $k_1$  $(x,y)$-dipaths and $k_2$ $(y,x)$-dipaths, all these dipaths being pairwise internally disjoint.
Recently, Cohen et al. showed that for every $(1,1)$- bispindle $B$, there exists an integer $k$ such that every strongly connected
digraph with chromatic number greater than $k$ contains a subdivision of $B$. We investigate generalisations of
this result by first showing constructions of strongly connected digraphs with large chromatic number without any $(3,0)$-bispindle
or $(2,2)$-bispindle. Then we show that strongly connected digraphs with large chromatic number contains a 
$(2,1)$-bispindle, where at least one of the $(x,y)$-dipaths and the $(y,x)$-dipath are long.    
\end{abstract}

\section{Introduction}
Throughout this paper, the {\it chromatic number} of a digraph $D$, denoted by $\chi(D)$, is the chromatic number of its underlying graph. 
In a digraph $D$, a \emph{directed path}, or \emph{dipath} is an oriented path where all the arcs are oriented in the same direction, from the initial vertex towards the terminal vertex.

A classical result due to Gallai, Hasse, Roy and Vitaver is the following. 
\begin{theorem}[Gallai~\cite{Gal68}, Hasse~\cite{Has64}, Roy~\cite{Roy67}, Vitaver~\cite{Vit62}]\label{thm:gallairoy}
	If $\chi(D) \geq k$, then $D$ contains a directed path of length $k+1$.  
\end{theorem}   

This raises the following question.

\begin{question}
Which digraphs are subdigraphs of all digraphs with large chromatic number ?
\end{question}

A famous theorem by Erd\H{o}s~\cite{Erd59} states that there exist graphs with arbitrarily high girth and arbitrarily large
chromatic number. This means that if $H$ is a digraph containing a cycle, there exist digraphs with arbitrarily high 
chromatic number with no subdigraph isomorphic to $H$. Thus the only possible candidates to generalise Theorem \ref{thm:gallairoy} are the {\it oriented trees} that are orientations of trees.
Burr\cite{Burr80} proved that every $(k-1)^2$-chromatic digraph contains every oriented tree of order $k$ and made the following conjecture.

\begin{conjecture}[Burr~\cite{Burr80}]\label{cnj:tn}
	 If $\chi(D) \geq (2k-2)$, then $D$ contains a copy of any oriented tree $T$ of order $k$.
\end{conjecture}
The best known upper bound, due to Addario-Berry et al.~\cite{AHS+13}, is in $(k/2)^2$.
However, for paths with two blocks ({\it blocks} are maximal directed subpaths), the best possible upper bound is known.

\begin{theorem}[Addario-Berry et al.~\cite{AHT07}]\label{thm:2blocks}
Let $P$ be an oriented path with two blocks on $n > 3$ vertices, then every digraph with chromatic number (at least) $n$ contains $P$. 
\end{theorem}

The following celebrated theorem of Bondy shows that  the story does not stop here.

\begin{theorem}[Bondy~\cite{Bon76}]\label{thm:bondy}
Every strong digraph of chromatic number at least $k$ contains a directed cycle of length at least $k$.
\end{theorem} 

The strong connectivity assumption is indeed necessary, as transitive tournaments contain no directed cycle but can have arbitrarily  high chromatic number. 

Observe that a directed cycle of length at least $k$ can be seen as a subdivision of $\vec{C}_k$, the directed cycle of length $k$.
Recall that a {\it subdivision} of a digraph $F$ is a digraph that can be obtained from $F$ by replacing each arc $uv$ by a directed path from $u$ to $v$. Cohen et al.  \cite{CHLN16} conjecture that Bondy's theorem can be extended to all oriented cycles.

\begin{conjecture}[Cohen et al.  \cite{CHLN16}]\label{conj:cycle-sub}
For every oriented cycle $C$, there exists a constant $f(C)$ such that every strong digraph with chromatic number at least $f(C)$ contains a subdivision of $C$.
\end{conjecture}

The strong connectivity assumption is also necessary in Conjecture~\ref{conj:cycle-sub} as shown by Cohen et al.  \cite{CHLN16}. This follows from  the following result.
\begin{theorem}\label{dkb}
For any positive integer $b$ and $k$, there exists an acyclic digraph $D_{k,b}$ such that any cycle in $D_{k,b}$ has at least $b$ blocks and $\chi(D_{k,b}) > k$.
\end{theorem}

On the other hand, Cohen et al.  \cite{CHLN16} proved conjecture for cycles with two blocks and the antidirected cycle of length $4$.
More precisely, denoting by $C(k,\ell)$ the cycle on two blocks, one of length $k$ and the other of length $\ell$,
 Cohen et al. \cite{CHLN16} proved the following result.
\begin{theorem}\label{th:ckl}
Every strong digraph with chromatic number at least $O((k+\ell)^4)$ contains a subdivision of $C(k,\ell)$.
\end{theorem} 
The bound has recently been improved to $O((k+\ell)^2)$  by Kim et al. \cite{KKPM}.

\medskip

A {\it $p$-spindle} is the union of $k$ internally disjoint $(x,y)$-dipaths for some vertices $x$ and $y$. Vertex $a$ is said to be the {\it tail} of the spindle and $b$ its {\it head}.
A {\it $(p+q)$-bispindle} is the internally disjoint union of a $p$-spindle
with tail $x$ and head $y$ and a  $q$-spindle with tail $y$ and head $x$.
In other words, it is the union of $k_1$  $(x,y)$-dipaths and $k_2$ $(y,x)$-dipaths, all of these dipaths being pairwise internally disjoint.
Note that $2$-spindles are the cycles with two blocks and the $(1+1)$-bispindles are the directed cycles.
In this paper, we generalize this and study the existence of spindles and bispindles in strong digraphs with large chromatic number.
First, we give a construction of digraphs with arbitrarily large chromatic number that contains no $3$-spindle and no $(2+2)$-bispindle. 
Therefore, the most we can expect in all strongly connected digraphs with large chromatic number are $(2+1)$-bispindle.
Let $B(k_1,k_2;k_3)$ denote the digraph formed by three internally disjoint paths between two vertices $x,y$,
two $(x,y)$-directed paths, one of size at least $k_1$, the other of size at least $k_2$, and one $(y,x)$-directed path of size at least $k_3$.
We conjecture the following.

\begin{conjecture}\label{conj:3chemins}
There is a function $g:\mathbb{N}^3 \rightarrow \mathbb{N}$ such that  every strong digraph with chromatic number at least $g(k_1,k_2,k_3)$ contains
a subdivision of $B(k_1,k_2;k_3)$.
\end{conjecture}

As an evidence, we prove this conjecture for $k_2=1$ and arbitrary $k_1$ and $k_3$. In Section~\ref{sec:k,1,1}, we first investigate the case $k_2=k_3=1$. We first prove in Proposition~\ref{prop:p211} that very strong digraph $D$ with $\chi(D)>3$ contains a subdivision of $B(2,1;1)$. We then prove the following.

\begin{theorem}\label{th:P11k}
Let $k \geq 3$ be an integer and let $D$ be a strong digraph. If $\chi(D) >  (2k-2)(2k-3)$, then $D$ contains a subdivision of $B(k,1;1)$.
\end{theorem}

In Section~\ref{sec:main}, using the same approach but in a more complicated way, we prove our main result: 

\begin{theorem}\label{th:main}
There is a constant $\gamma_k$ such that if $D$ is a strong digraph with $\chi(D) > \gamma_k$, then $D$ contains a subdvision of $B(k,1;k)$.
\end{theorem}

We prove the above theorem for a huge constant $\gamma_k$. It can easily be lowered. However, we made no attempt to it here for two reasons:
firstly, we would like to keep the proof as simple as possible; secondly using our method, there is no hope to get an optimal or near optimal value for $\gamma_k$.

\medskip

Similar questions with $\chi$ replaced by another graph parameter can be studied.
We refer the reader to \cite{AC+16} and \cite{CHLN16} for more exhaustive discussions on such questions.
Let us just give one result proved by Aboulker et al. \cite{AC+16} which can be seen as an analogue to Conjecture~\ref{conj:3chemins}.

\begin{theorem}[Theorem 28 in \cite{AC+16}]
Let $k_1,k_2,k_3$ be  positive integers with $k_1 \ge k_2$.
Let $D$ be a digraph with $\delta^+(D) \geq 3k_1+2k_2+k_3-5$. Then $D$ contains $B(k_1,k_2;k_3)$ as a subdivision.
\end{theorem}

\section{Definitions and preliminaries}

We follow standard terminology as used in \cite{BoMu08}. We denote by $[k]$ the set of integers $\{1, \dots , k\}$.

\smallskip

Let $F$ be a digraph.
A digraph $D$ is said to be {\it $F$-subdivision-free}, if it contains no subdivision of $F$.

\smallskip

The {\it union} of two digraphs $D_1$ and $D_2$ is the digraph  $D_1\cup D_2$ defined by $V(D_1\cup D_2) = V(D_1)\cup V(D_2)$ and 
$A(D_1\cup D_2) = A(D_1)\cup A(D_2)$.
If ${\cal D}$ is a set of digraphs, we denote by $\bigcup {\cal D}$ the union of the digraphs, i.e. $V(\bigcup {\cal D}) =\bigcup_{D\in {\cal D}} V(D)$
and $A(\bigcup {\cal D}) =\bigcup_{D\in {\cal D}} A(D)$.

\smallskip

Let $P$ be a path. We denote by $s(P)$ its initial vertex and by $t(P)$ its terminal vertex.
If $D$ is a directed path or a directed cycle, then we denote by $D[a,b]$ the subdipath of $D$ with initial vertex $a$ and terminal vertex $b$.
We denote by $D[a,b[$ the dipath $D[a,b] -b$, by $C]a,b]$ the dipath $D[a,b] -a$, and by $D]a,b[$ the dipath $D]a,b[ -\{a,b\}$.
If $P$ and $Q$ are two directed paths such that $V(P)\cap V(Q) =\{s(P)\}=\{t(Q)\}$, the {\it concatenation} of $P$ and $Q$, denoted by $P\odot Q$, is the dipath $P\cup Q$. 

\smallskip


A digraph is {\it connected} (resp. {\it $2$-connected}) if its underlying graph is connected (resp. $2$-connected.
The {\it connected components} of a digraph are the connected components of its underlying graph.
A digraph $D$ is {\it strongly connected} or {\it strong} if for any two vertices $x,y$ there is directed path from $x$ to $y$.
The {\it strong components} of a digraph are its maximal strong subdigraphs.

\smallskip

Let $G$ be a graph or a digraph.
A {\it proper $k$-colouring} of $G$ is a mapping $\phi: V(G) \rightarrow [k]$ such that $\phi(u) \neq \phi(v)$ whenever $u$ is adjacent to $v$. 
$G$ is {\it $k$-colourable} if it admits a proper $k$-colouring. The {\it chromatic number} of $G$, denoted by $\chi(G)$, is the least integer $k$
such that $G$ is $k$-colourable.

A (directed) graph $G$ is {\it $k$-degenerate} if every subgraph $H$ of $G$ has a vertex of degree at most $k$.
The following proposition is well-known.
\begin{proposition}\label{prop:deg}
Every $k$-degenerate (directed) graph is $(k+1)$-colourable.
\end{proposition}

\begin{theorem}[Brooks]\label{thm:brooks}
Let $G$ be a connected graph.
Then $\chi(G)\leq \Delta(G)$ unless $G$ is a complete graph or an odd cycle.
\end{theorem}

The following easy lemma is well-known.

\begin{lemma}\label{lem:decomp}
Let $D_1$ and $D_2$ be two digraphs.
$\chi(D_1\cup D_2) \leq \chi(D_1)\times \chi(D_2)$.
\end{lemma}

\begin{lemma}\label{lem:contrac}
Let $D$ be a digraph, $D_1 \dots D_l$ be disjoint subdigraphs of $D$ and $D'$ the digraph obtained by contracting each $D_i$ into
one vertex $d_i$. Then $\chi(D) \leq \chi(D')\cdot \max\{\chi(D_i) \mid i \in [l]\}$.
\end{lemma}

\begin{proof}

Set $k_1 = \max\{\chi(D_i) \mid i \in [l]\}$ and $k_2 = \chi(D')$. For each $i$, let $\phi_i$ be a proper colouring of $D_i$ using colours in $[k_1]$ and let
$\phi'$ be a proper colouring of $D'$ using colours in $[k_2]$. 
Define $\phi : V(D) \rightarrow [k_1] \times [k_2]$ as follows. If $x$ is a vertex belonging to some $D_i$, then $\phi(x) = (\phi_i(x), \phi'(d_i))$, else $\phi(x) =(1,\phi'(x))$. 
Let $x$ and $y$ be adjacent vertices of $D$. If they belong to the same subdigraph $D_i$, then $\phi_i(x) \not = \phi_i(y)$ and so $\phi(x) \not = \phi(y)$. If they do not belong
to the same component, then the vertices corresponding to these vertices in $D_{\mathcal{C}}$ are adjacent and so $\phi(x) \not = \phi(y)$. 
Thus $\phi$ is a proper colouring of $D$ using $k_1\cdot k_2$ colours. 
\end{proof}

The {\it rotative tournament on $2k-1$ vertices}, denoted by $R_{2k-1}$, is the tournament with vertex set $\{v_1, \dots , v_{2k-1}\}$ in which $v_i$ dominates $v_j$ if and only if $1\leq j-i \leq k-1$ (indices are modulo $2k-1$).

\begin{proposition}\label{prop:tournoi}
Every strong tournament of order $2k-1$ contains a $B(k,1;1)$-subdivision.
\end{proposition}
\begin{proof}
Let $T$ be a strong tournament of order $2k-1$. By Camion's Theorem, it has a hamiltonian directed cycle $C=(v_1,v_2, \dots ,v_{2k-1},v_1)$.
If there exists an arc $v_iv_j$ with $j-i\geq k$  (indices are modulo $2k-1$), then the union of  $C[v_i,v_j]$, $(v_i,v_j)$ and $C[v_j,v_i]$ is a $B(k,1;1)$-subdivision.
Henceforth, we may assume that $T=R_{2k-1}$.
Then the union of $C[v_1, v_{k-1}] \odot (v_{k-1}, v_{k+1}, v_{k+2})$,  $(v_1, v_{k}, v_{k+2})$, and $C[v_{k+2}, v_1]$ is a $B(k,1;1)$-subdivision.
\end{proof}

Let $F$ be a subdigraph of a digraph $D$. A {\em directed
  ear} of $F$ in $D$ is a directed path in $D$ whose ends lie in $F$ but whose internal
vertices do not.   The following lemma is well known.

\begin{lemma}[Proposition 5.11 in \cite{BoMu08}] \label{lem:diear}
Let $F$ be a nontrivial proper $2$-connected strong subdigraph of a
$2$-connected strong digraph $D$. Then $F$ has a directed ear in $D$.
\end{lemma}

We will need the following lemmas:

%
%

\begin{lemma}\label{min}
Let $\sigma=(u_t)_{t\in [p]}$ be a sequence of integers in $[k]$, and let $l$ be a positive integer. If $p\geq l^k$, then there exists a set $L$ of $l$ indices such that for any $i,j \in L$ with $i < j$ the following holds : $u_i=u_j$ and $u_t > u_i$, for all $i < t < j$. \end{lemma}

\begin{proof}
By induction on $k$, the result holding trivially when $k=1$. Assume now that $k>1$. Let $L_1$ be the elements of the sequence with value $1$. If $L_1$ has at least $l$ elements, we are done.
If not, then there is a subsequence $\sigma'$ of $\left\lceil \frac{l^k-(l-1)}{l}\right \rceil = l^{k-1}$ consecutive elements in $\{2, \dots , k-1\}$. Applying the induction hypothesis to $\sigma'$ yields the result.
\end{proof}

\begin{lemma}\label{max}
Let $\sigma=(u_t)_{t\in [p]}$ be a sequence of integers in $[k]$. 
If $p > k (m-1)$, then there exists a subsequence of $m$ consecutive integers such that the
last one is the largest.
\end{lemma}

\begin{proof}
By induction on $k$, the result holding trivially when $k=1$. 
Let $i$ be the smallest integer such that $u_t\leq k-1$ for all $t\geq i$.
If $i>m$, then $u_{i-1}=k$, and the subsequence of the $i-1$ first elements of $\sigma$ is the desired sequence.
If $i\leq m$, apply the induction on $\sigma'=(u_t)_{i\leq t\leq p}$ which is a sequence of more than $(k-1)(m-1)$ integers in $[k-1]$, to get the result. 
\end{proof}

\section{$3$-spindles and $(2+2)$-bispindles}



\begin{theorem}
For every integer $k$, there exists a strongly connected digraph $D$ with $\chi(D) >k$ that contains no $3$-spindle and no $(2+2)$-bispindle. 
\end{theorem}

\begin{proof}
Let $D_{k,4}$ an acyclic digraph with chromatic number greater than $k$ and with every cycle having at least four blocks obtained by Theorem \ref{dkb}.
Let $S = \{s_1, \dots,  s_l\}$ be the set of vertices of $D_{k,4}$ with out-degree 0 and $T = \{t_1, \dots, t_m\}$ the set of vertices with in-degree 0.

Consider the digraph $D$ obtained from $D_{k,4}$ as follows. Add a dipath  $P = (x_1,x_2, \dots , x_l,z,y_1,y_2,\dots, y_m)$ and the arc $s_ix_i$ for all $i\in [l]$ and
$y_jt_j$ for all $j\in [m]$. It is easy to see that $D$ is strong.
Moreover, in $D$, every directed cycle uses the arc $x_lz$. Therefore $D$ does not contain a $(2+2)$-bispindle, which has two arc-disjoint directed cycles. 

Suppose now that $D$ has a $3$-spindle with tail $u$ and head $v$, and let $Q_1, Q_2, Q_3$ be its three $(u,v)$-dipaths. Observe that $u$ and $v$ are not vertices of $P$, because all vertices of this dipath have either in-degree $2$ or out-degree $2$. In $D$ each cycle with two blocks between vertices outside $P$ must use the arc $x_lz$. The union of $Q_1$ and $Q_2$ form a cycle on two blocks, 
which means one of the two paths, say $Q_1$, contains $x_lz$. But $Q_2$ and $Q_3$ also form a cycle on two blocks, but they cannot contain $x_lz$, a contradiction.
\end{proof}

\section{$B(k,1;1)$}\label{sec:k,1,1}

\begin{proposition}\label{prop:p211}
Let $D$ be a strong digraph.
If $\chi(D)\geq 4$, then $D$ contains a $B(2,1;1)$-subdivision.
\end{proposition}
\begin{proof}
Assume  $\chi(D)\geq 4$. Since every digraph contains a $2$-connected strong subdigraph with the same chromatic number, we may assume that $D$ is $2$-connected.
Let $C$ be a shortest directed cycle in $D$. 
It must be induced, so $\chi(D[C])=\chi(C) \leq 3$. 
Now by Lemma~\ref{lem:diear}, $C$ has a directed ear $P$ in $D$. Necessarily, $P$ has length at least $2$ since $C$ is induced.
Thus the union of $P$ and $C$ is a $B(2,1;1)$-subdivision.
\end{proof}

The bound $4$ in Proposition~\ref{prop:p211} is best possible because a directed odd cycle has chromatic number $3$ and contains no $B(2,1;1)$-subdivision.

\medskip

In the remaining of this section, we present a proof of Theorem~\ref{th:P11k}.

Let  ${\cal C}$ be a collection of directed cycles. It is {\it nice} if all cycles of ${\cal C}$ have length at least $2k-2$, and
any two distinct cycles $C_i,C_j\in\mathcal C$ intersect on at most one vertex. 
A {\it component} of $\mathcal{C}$ is a connected component in the adjacency graph of $\mathcal{C}$, where vertices correspond to cycles in $\mathcal{C}$
and two vertices are adjacent if the corresponding cycles intersect. Note that if ${\cal S}$ is a component of ${\cal C}$, then $\bigcup{\cal S}$ is 
both a connected component and a strong component of $\bigcup{\cal C}$.
Call $D_{\mathcal{C}}$ the digraph obtained from $D$ by contracting each component of $\mathcal{C}$ into one vertex. 
For sake of simplicity, we denote by $D[{\cal S}]$ the digraph $D[\bigcup {\cal S}]$. Observe that this digraph contains $\bigcup {\cal S}$ but has more arcs.

We will prove that every $B(k,1;1)$-subdivision-free strong digraph $D$ has bounded chromatic number in the following way:
Take ${\cal C}$ a maximal nice collection of cycles. 
We will prove that every component $S$ of ${\cal C}$ induces a digraph $D[{\cal S}]$ on $D$ of bounded chromatic number. 
Then we will prove that, since it contains no long directed cycle and it is strong, $D_{\mathcal{C}}$ has bounded chromatic number,
which, by Lemma~\ref{lem:contrac}, allows us to conclude. 

We will need the following lemma:

\begin{lemma}\label{lem:headphone}
Let ${\cal C}$ be a nice collection of cycle in a $B(k,1;1)$-subdivision-free digraph $D$ and let $C$, $C'$ be two cycles of the same
component ${\cal S}$ of $\mathcal{C}$. There is no dipath $P$ from $C$ to $C'$ whose arcs are not in $A (\bigcup  {\cal S})$. 
\end{lemma}

\begin{proof}
By the contrapositive. We suppose that there exists such a dipath $P$ and show that there is a subdivision of $B(k,1;1)$ in $D$.

By definition of ${\cal S}$, there exists a dipath $Q$ from $C$ to $C'$ in $\bigcup S$. By choosing $C$ and $C'$ such that 
$Q$ is as small as possible, then $s(Q) \not = t(P)$ and $t(Q) \not = s(P)$ (note that $s(Q)$ and $t(Q)$ can be the same vertex). 

Since $C$ has length at least $2k-2$, either $C[t(Q), s(P)]$ has length at least $k-1$ or $C[s(P), t(Q)]$ has length at least $k$. 

\begin{itemize}
	\item If $C[t(Q), s(P)]$ has length at least $k-1$, then the union of $Q \odot C[t(Q), s(P)] \odot P$, $C'[s(Q), t(P)]$ and $C'[ t(P), s(Q)]$
	is a subdivision of $B(k,1;1)$ between $s(Q)$ and $t(P)$. 
	\item If $C[s(P), t(Q)]$ has length at least $k$, then the union of $C[s(P), t(Q)]$, $P \odot C'[t(P), s(Q)] \odot Q$ and $C[t(Q), s(P)]$
	is a subdivision of $B(k,1;1)$ between $s(P)$ and $t(Q)$. 
\end{itemize}
\end{proof}

\begin{lemma}\label{lem:compo1}
Let $k\geq 3$ be an integer, and let ${\cal C}$ be a nice collection of cycles in a $B(k,1;1)$-subdivision-free digraph $D$ and ${\cal S}$ a component of $\mathcal{C}$. Then $\chi(D[S]) \leq 2k-2$.	 
\end{lemma}

\begin{proof}
By induction on the number of cycles in ${\cal S}$. Let $C$ be a cycle of ${\cal S}$. There is no chord between $x$ and $y$ in $C$
such that $C[x,y]$ has length at least $k$, for otherwise there would be a $B(k,1;1)$-subdivision. Hence $D[C]$ has maximum degree at most $2k-2$. Moreover, by Proposition~\ref{prop:tournoi}, $D[C]$ is not a tournament of order $2k-1$.
Thus, by Brooks' Theorem (\ref{thm:brooks}), $\chi(D[C])\leq 2k-2$.
 Let $c$ be a proper colouring of $C$ with $2k-2$ colours. Let ${\cal S}_1, {\cal S}_2, \dots, {\cal S}_r$ be the components of ${\cal S} \setminus{C}$. Since ${\cal S}$ is the union the ${\cal S}_l$, $l\in [r]$, and $\{C\}$,
each ${\cal S}_l$ has less cycles than ${\cal S}$. By the induction hypothesis, there exists a proper colouring $c_l$ using $2k-2$ colours for each $D[{\cal S}_i]$.
 
Now, we claim that each $D[{\cal S}_l]$ intersects $C$ in exactly one vertex. It is easy to see that $C$ must intersect at least one cycle of each ${\cal S}_l$.
Now suppose there exist two vertices of $C$, $x$ and $y$ in $D[{\cal S}_l]$. By definition of a nice collection, they cannot belong to the same cycle 
of ${\cal S}_l$, so there exist two cycles $C_i$ and $C_j$ of $S_l$ such that $x \in C_i$ and $y \in C_j$. Now $C[x,y]$ is a dipath form $C_i$ to $C_j$ whose arcs are not in $A(\bigcup {\cal S}_l)$. This contradicts Lemma~\ref{lem:headphone}. 

Consequently, free to permute the colours of the $c_l$,
we may assume that each vertex of $C$ receives the same colour in $c$ and in the $c_l$. 
In addition, by Lemma~\ref{lem:headphone}, there is no arc between different $D[{\cal S}_l]$ nor between $D[{\cal S}_l]$ and $C$. Hence the union of the 
$c_l$ and $c$ is a proper colouring of $D[{\cal S}]$ using $2k-2$ colours. 
\end{proof}

\begin{lemma}\label{lem:reduce}
Let ${\cal C}$ be a maximal nice collection of cycle in a $B(k,1;1)$-subdivision-free strong digraph $D$. Then $\chi(D_{\mathcal{C}}) \leq 2k-3$
\end{lemma}

\begin{proof}
First note that since $D$ is strong, then so is $D_{\mathcal{C}}$. Suppose $\chi(D_{\mathcal{C}}) \geq 2k-2$. By Bondy's Theorem, there exists a directed cycle
$C = x_1\dots x_l$ of length at least $2k-2$ in $ D_{\mathcal{C}}$. We derive a cycle $C'$ in $D$ the following way:
Suppose the vertex $x_i$ corresponds to a component ${\cal S}_i$ of ${\cal C}$: the arc $x_{i-1}x_i$ corresponds in $D$ to an arc whose head
is a vertex $p_i$ of $\bigcup {\cal S}_i$, and the arc $x_ix_{i+1}$ corresponds to an arc whose tail is a vertex  $l_i$ of $\bigcup {\cal S}_i$. 
Let $P_i$ be a dipath 
from $p_i$ to $l_i$ in $D[{\cal S}_i]$. Note that $P_i$ intersects each cycle of $S_i$ on a, possibly empty, subpath of $P_i$. 
Then $C'$ is the cycle obtained from $C$ by replacing the vertices $x_i$ by the path $P_i$.

$C'$ is a cycle of $D$ of length at least $2k-2$ because it is no shorter than $C$. Let $C_1$ be a cycle of $\mathcal{C}$. By construction of $C'$ 
and $D_{\mathcal{C}}$, $C'$ and $C_1$ can intersect only along a subpath of one $P_i$. Suppose this dipath is more than just one vertex. Let $x$ and
$y$ the initial and terminal vertex of this path. Then the union of $C'[x,y]$,
$C_1[x,y]$ and $C_1[y,x]$ is a $B(k,1;1)$-subdivision.

So $C'$ is a cycle of length at least $2k-2$, intersecting each cycle of $\mathcal{C}$ on at most one vertex, and which does not belong to $\mathcal{C}$, for otherwise would be reduced
to one vertex in $D_{\mathcal{C}}$. This contradicts the fact that $\mathcal{C}$ is maximal.
\end{proof}

So we can finally prove Theorem \ref{th:P11k}.

\begin{proof}[Proof of Theorem~\ref{th:P11k}]
Let ${\cal C}$ be a maximal nice collection of cycle in $D$.  Lemmas \ref{lem:compo1}, \ref{lem:reduce} and \ref{lem:contrac} give the result.
\end{proof}

\section{$B(k,1;k)$}\label{sec:main}

In this section, we present a proof of Theorem \ref{th:main}.

We prove the result by the contrapositive. We consider a digraph $D$ that contains no subdivision of $B(k,1;k)$.
We shall prove that $\chi(D) \leq  C_k= \cste$.

Our proof heavily uses the notion of $k$-suitable collection of directed cycles, which can be seen as a generalization of the notion of nice collection of cycles used to prove Theorem~\ref{th:P11k}.

 A collection ${\cal C}$ of directed cycles is {\it $k$-suitable} if all cycles of ${\cal C}$ have length at least $8k$, and
any two distinct cycles $C_i,C_j\in\mathcal C$ intersect on  a subpath $P_{i,j}$ of order at most $k$.
We denote by $s_{i,j}$ (resp. $t_{i,j}$) the initial (resp. terminal) vertex of  $P_{i,j}$. 
the notion of nice collection seen before. 

The proof of Theorem \ref{th:main} uses the same idea as Theorem \ref{th:P11k}: take a maximal $k$-suitable collection of directed cycles ${\cal C}$; show that the digraph $D_{\cal C}$ obtained by contracting the components of ${\cal C}$ has bounded chromatic number, and that each component also has bounded chromatic; conclude using Lemma~\ref{lem:contrac}. However, because the intersection of cycles in this collection are more
complicated and because there might be arcs between cycles of the same component, bounding the chromatic number of the components is way more challenging. The next subsection is devoted to this.

 \subsection{$k$-suitable collections of directed cycles}

 \begin{lemma}\label{lem:dis}
 Let ${\cal C}$ be a $k$-suitable collection of directed cycles in a $B(k,1;k)$-\free\ digraph.
Let $C_1,C_2,C_3\in\mathcal C$ which pairwise intersect, and let $v$ belong to $V(C_2)\cap V(C_3)\setminus V(C_1)$. 
Then exactly one of the following holds:
\begin{itemize}
\item[(i)] $C_2[t_{1,2}, v]$ and $C_3[t_{1,3}, v]$ have both length less than $3k$; 
\item[(ii)] $C_2[v, s_{1,2}]$ and $C_3[v, s_{1,3}]$ have both length less than $3k$.
\end{itemize}
\end{lemma}

\begin{proof}
Observe first that since $C_2$ has length at least $8k$ and $P_{1,2}$ has length at most $k-1$, the sum of lengths of $C_2[t_{1,2}, v]$ and  $C[v, s_{1,2}]$ is at least $7k+1$. Similarly, the sum of lengths of $C_2[t_{1,3}, v]$ and $C[v, s_{1,3}]$ is at least $7k+1$. In particular, if (i) holds, then (ii) does not hold and vice-versa.

Suppose for a contradiction that both (i) and (ii) do not hold. By symmetry and the above inequalities, we may assume that both $C_2[t_{1,2}, v]$ and $C_3[v, s_{1,3}]$ have length more than $3k$.
But $v\notin V(C_1)$, so $v\notin V(P_{1,3})$. Thus $C_3[v, t_{1,3}]$  has also length at least $3k$. 

If there is a vertex in $V(C_1)\cap V(C_2)\cap V(C_3)$, then $C_3[v,t_{1,3}]$ would have length less than $2k$ (since it would be contained in $P_{2,3}\cup P_{1,3}$ and each of those paths has length less than $k$), a contradiction.
Hence $V(C_1)\cap V(C_2)\cap V(C_3)=\emptyset$. In particular, $P_{1,2}$, $P_{1,3}$, and $P_{2,3}$ are disjoint.

The dipath  $C_2[s_{1,2}, t_{2,3}]$ has length at least $3k$ because it contains $C_2[t_{1,2},v]$. Moreover, the dipath $C_3[t_{2,3}, s_{1,3}]$ 
has length at least $2k$ because $C_3[v, s_{1,3}]$ has length at least $3k$ and $C_3[v,t_{2,3}]$ has length less than $k$.
Thus $C_3[t_{2,3}, s_{1,3}]\odot C_1[s_{1,3}, s_{1,2}]$ has length at least $2k$.
Consequently, the union of $C_2[s_{1,2}, t_{2,3}]$, $C_2[t_{2,3}, s_{1,2}]$, and $C_3[t_{2,3}, s_{1,3}]\odot C_1[s_{1,3}, s_{1,2}]$  is a subdivision of $B(k,1;k)$, a contradiction.
\end{proof}





Let ${\cal C}$ be a $k$-suitable collection of directed cycles.
For every set of vertices or digraph $S$, we denote by ${\cal C}\cap S$ the set of cycles of ${\cal C}$ that intersect $S$.

Let $C_1\in \mathcal{C}$.
For each $C_j\in {\cal C} \cap C_1$, let $Q_j$ be the subdipath of $C_j$ 
containing all the vertices that are at distance at most $3k$ from $P_{1,j}$ in the cycle underlying $C_j$.
Then the dipath $C_j[s(Q_j), s_{1,j}]$ and  $C_j[t_{1,j},t (Q_j)]$ have length $3k$.
Set $Q^-_j=C[s(Q_j), s_{1,j}[$ and $Q^+_j=C]t_{1,j}, t(Q_j)]$.

Set  $I(C_1) =C_1\cup \bigcup_{C_j\in {\cal C}\cap C_1} Q_j$, $I^+(C_1) = \bigcup_{C_j\in {\cal C}\cap C_1} Q^+_j$ and $I^-(C_1) = \bigcup_{C_j\in {\cal C}\cap C_1} Q^-_j$.
Observe that Lemma~\ref{lem:dis} implies directly the following.
\begin{corollary}\label{cor:util}
Let ${\cal C}$ be a $k$-suitable collection of directed cycles and let $C_1\in {\cal C}$.
\begin{itemize}
\item[(i)]  $I^+(C_1)$ and $I^-(C_1)$ are vertex-disjoint digraphs. 
\item[(ii)] $I^-(C_1) \cap C_j = Q^-_j$ and $I^+(C_1) \cap C_j = Q^+_j$, for all $C_j\in {\cal C}\cap C_1$.
\end{itemize}
 \end{corollary}

\begin{lemma}\label{lem:A}
Let ${\cal C}$ be a $k$-suitable collection of directed cycles in a $B(k,1;k)$-\free\ digraph $D$.
Let $C_1$ be a cycle of ${\cal C}$ and let $A$ be a connected component of $\bigcup{\cal C} - I(C_1)$.
All vertices of $\bigcup({\cal C}\cap A)  - A$ belongs to a unique cycle $C_A$ of $\mathcal{C}$.
\end{lemma}

\begin{proof}
Suppose it is not the case. Then there are two distinct cycles $C_2, C_3$ of ${\cal C}\cap A$ that intersect with $C_1$.
Observe that there is a sequence of distinct cycles $C_2=C^*_1, C^*_2, \dots , C^*_q=C_3$ of cycles of ${\cal C}\cap A$ such that $C^*_j\cap C^*_{j+1}\neq \emptyset$ because $A$ is a connected component of $\bigcup{\cal C} - I(C_1)$.
Free to consider the first $C^*_j \neq C_2$ in this sequence such that $V(C^*_j)\not\subseteq A$ in place of $C_3$, we may assume
 that all $C^*_j$, $2\leq j\leq q-1$, have all their vertices in $A$.
 In particular, there exist a  $(C_3,C_2)$-dipath $Q_A$ in $D[A]$.

Let $R_3=C_1[t_{1,2},t_{1,3}]\odot Q_3$. Clearly, $R_3$ has length at least $3k$.
Let $v$ be the last vertex in $Q_2\cap R_3$ along $Q_2$. (This vertex exists since $t_{1,2}\in Q_2\cap R_3$.) Since there is a $(C_3,C_2)$-dipath in $D[A]$, by Corollary~\ref{cor:util}, $C_3[t(Q_3), s(Q_A)]$ is in $D[A]$. Thus  there exists a $(t(Q_3), C_2)$-dipath $R_A$ in $D[A]$.  Let $w$ be its terminal vertex. By definition of $A$, $w$ is in $C_2[t(Q_2), s(Q_2)]$, therefore $C_2[w,v]$ has length at least $3k$ since it contains $C_2[s(Q_2), s_{1,2}]$. Consequently, both $C_2[v,t(Q_2)]$ and $R_3[v,t(Q_3)]$ have length less than $k$ for otherwise the union of $C_2[w,v]$, $C_2[v,w]$ and $R_3[v,t(Q_3)]\odot R_A$ would be a subdivision of $B(k,1;k)$.
In particular, $v\neq t(Q_2)$. This implies that $s_{2,3}\in V(Q_2\cap R_3)$.
Moreover, $Q_2[s_{2,3}, t(Q_2)]$ has length less than $2k$ because $Q_2[s_{2,3},v]$ is a subdipath of $P_{2,3}$ and so has length less than $k$.
Therefore $C_2[t_{1,2},s_{2,3}]=Q_2[t_{1,2},s_{2,3}]$ has length at least $k$ because $Q_2$ has length at least $3k$.
It follows that the union of $C_2[s_{2,3},t_{1,2}]$,  $C_2[t_{1,2},s_{2,3}]$ and  $R_3[t_{1,2},s_{2,3}]$ is a subdivision of $B(k,1;k)$, a contradiction.
\end{proof}

\begin{lemma}\label{lem:no-dicycle}
Let ${\cal C}$ be a $k$-suitable collection of directed cycles in a $B(k,1;k)$-\free\ digraph. For any cycle $C_1\in {\cal C}$, the digraph
$I^+(C_1)$ has no directed cycle.  
\end{lemma}

\begin{proof}

Suppose for a contradiction that $I^+(C_1)$  contains a directed cycle $C'$.
Clearly, it must contain arcs from at least two $Q^+_j$.

Assume that $C'$ contains several vertices of $Q^+_j$.
Necessarily, there must be two vertices $x,y$ of $Q^+_j\cap C'$ such that no vertex of $C']x,y[$ is in $C_j$ and  $y$ is before $x$ in $Q^+_j$.
Therefore $C'[x,y]\odot Q^+[y,x]$ is also a directed cycle in $I^+(C_1)$. Free to consider this cycle,
we may assume that $C'\cap Q^+_j$ is a directed path.

Doing so, for all $j$, we may assume that $C'\cap Q^+_j$ is a directed path  for every $C_j\in {\cal C}\cap C_1$.
Without loss of generality, we may assume that there are cycles $C_2, \dots , C_p$ such that
\begin{itemize}
\item $C'$ is in $Q^+_2\cup \cdots  \cup Q^+_p$;
\item for all $2\leq j\leq p$, $C'\cap Q^+_j$ is a directed path $P^+_j$ with initial vertex $a_j$ and terminal vertex $b_j$;
\item the $a_j$ and the $b_j$ appear according to the following order around $C'$: $(a_2, b_p, a_3, b_2, \dots ,   a_p, b_{p-1}, a_2)$ with possibly $a_{j+1}=b_j$ for some $1\leq j \leq p$ where $a_{p+1}=a_2$.
\end{itemize}
For $2\leq j\leq p$, set $B_j=C_j[b_j, a_j]$. Note that $B_j$ has length at least $4k$, because $Q^+_2$ has length less than $3k$.

Consider the closed directed walk $$W=C_p[a_2,b_p]\odot B_p \odot C_{p-1}[a_p, b_{p-1}] \odot  \cdots  \odot  B_3\odot C_2[a_3, b_2]\odot B_2.$$
$W$ contains a directed cycle $C_W$. Wihtout loss of generality, we may assume that this cycle is of the form
$$C_W=B_q[v, a_q] \odot C_{q-1}[a_q, b_{q-1}] \odot  \cdots  \odot  B_3\odot C_2[a_3, b_2]\odot B_2[b_2, v]$$
for some vertex $v\in B_2\cap B_q$. (The case when $W$ is a directed cycle corresponds to $q=p+1$ and $B_2=B_{p+1}$.)

Note that necessarily, $q\geq 4$, for $B_3$ does not intersect $B_2$, for otherwise $b_3=b_2$ since the intersection of $C_2$ and $C_3$ is a dipath.

Observe that $C_W[b_2,v]=C_2[b_2,v]$ or $C_W[v, a_4]$ has length at least $k$.
Indeed, if $q=p+1$, then it follows from the fact that $B_2$ has length as least $4k$; if
$5\leq q\leq p$, then it comes form the fact that $B_4$ is a subdipath of $C_W[v, a_r]$; if $q=4$, then it follows from Lemma~\ref{lem:dis}  applied to $C_3$, $C_2$, $C_4$ in the role of $C_1$, $C_2$, $C_3$ respectively.
In both case, $C_W[b_2, a_4]$ has length at least $k$. 

Furthermore, $C_W[a_4,b_2]$ has length at least $k$ because it contains $B_3$. Therefore the union of
$C_W[b_2, a_4]$, $C_W[a_4,b_2]$ and $C'[b_2,a_4]=C_3[b_3,a_4]$ is a subdivision of of $B(k,1;k)$, a contradiction.
\end{proof}

Let $\phi$ be a colouring of $G$. A subset of vertices or a subgraph $S$ of $G$ is {\it rainbow-coloured} by $\phi$ if all vertices of $S$ have distinct colours.

Set $\alpha_k=2\cdot \col + 14k$.

\begin{lemma}\label{lem:IC}
Let ${\cal C}$ be a $k$-suitable collection of directed cycles in a $B(k,1;k)$-\free\ digraph.

Let $\phi$ be a partial colouring of a cycle $C_1\in {\cal C}$ such that only a path of length at most
$7k$ is coloured and this path is rainbow-coloured. Then $\phi$ can be extended into
a colouring of $I(C_1)$ using $\alpha_k$ colours, such that every subpath of length at most $7k$ of $C_1$ is rainbow-coloured and  $Q_j$ is rainbow-coloured, for every $C_j\in {\cal C}\cap C_1$.\end{lemma}

\begin{proof}
We can easily extend $\phi$ to $C_1$ using $14k$ colours (including the at most $7k$ already used colours)
so that every subpath of $C_1$ of length $7k$ is rainbow-coloured.

We shall now prove that there exists a colouring $\phi^+$ of $I^+(C_1)$ with $\col$ (new) colours so that $Q^+_j$ is rainbow-coloured for every $C_j\in {\cal C}\cap C_1$, and a colouring $\phi^-$ of $I^-(C_1)$ with $\col$ (other new) colours so that $Q^-_j$ is rainbow-coloured for every $C_j\in {\cal C}\cap C_1$.
The union of the three colourings $\phi$, $\phi^+$, and $\phi^-$ is clearly the desired colouring of $I(C_1)$. (Observe that a vertex of $I(C_1)$ is coloured only once because  $C_1$, $I^+(C_1)$ and $I^-(C_1)$ are disjoint by Corollary~\ref{cor:util}.)

\smallskip

It remains to prove the existence of $\phi^+$ and $\phi^-$. By symmetry, it suffices to prove the existence of $\phi^+$.
 To do so, we consider an auxiliary digraph $D^+_1$.
 For each $C_j\in {\cal C}\cap C_1$, let $T^+_j$ be the transitive tournament whose hamiltonian dipath is $Q^+_j$.
Let $D^+_1=\bigcup_{C_j\in  {\cal C}\cap C_1} T^+_j$. 
The arcs of the $A(T^+_j)\setminus A(Q^+_j)$ are called {\it fake arcs}.
Clearly, $\phi^+$ exists if and only if $D^+_1$ admits a proper $\col$-colouring. Henceforth it remains to prove the following claim.

\begin{claim}\label{claim:c+}
$\chi(D^+_1) \leq  \col$.
\end{claim}

\begin{subproof}
To each vertex $v$ in $I^+(C_1)$ we associate the set $\Dis(v)$ of the lengths of the $C_j[t_{1,j},v]$ for all cycle $C_j\in {\cal C} \cap C_1$ containing $v$ such that $C_j[t_{1,j},v]$ has length at most $3k$.

Suppose for a contradiction that $\chi(D^+_1) \leq  \col$.
By Theorem \ref{thm:gallairoy}, $D^+_1$ admits a directed path of length $\col$. Replacing all fake arcs  $(u,v)$ in some $A(T^+_j)$, by 
$Q^+_j[u,v]$ we obtain a directed walk $P$ in $I^+(C_1)$ of length at least $\col$. By Lemma \ref{lem:no-dicycle}, $P$ is necessarily a directed path. Set $P=(v_1,  \dots , v_{p})$. We have $p\geq \col$.

For $1\leq i\leq p$, let $m_i = \min \Dis(v_i)$. 
Lemma \ref{min} applied to $(m_i)_{1\leq i\leq p}$ yields a set $L$ of $6k^2$ indices of such that 
for any $i< j \in L$,  $m_i=m_j$ and $m_k > m_i$, for all $i< k < j$.
Let $l_1 < l_2 < \cdots < l_{6k^2}$ be the elements of $L$ and let $m= m_{l_1} = \cdots = m_{l_{6k^2}}$.

For $1\leq j\leq 6k^2-1$, let $M_j = \max \bigcup_{l_j\leq i < l_{j+1}} \Dis(v_i)$.
By definition $M_j\leq 3k$.
Applying Lemma~\ref{max} to $(M_j)_{1\leq j\leq 6k^2}$,  we get a sequence of size $2k$ $M_{j_0+1} \dots M_{j_0+{2k}}$ such that $M_{j_0+{2k}}$ is the greatest. For sake of simplicity, we set $\ell_i =j_0+i$ for $1\leq i\leq 2k$.
Let $f$ the smallest index not smaller than $\ell_{2k}$ for which $M_{\ell_{2k}} \in \Dis (v_f)$. 

Let $j_1$ be an index such that $C_{j_1}[t_{1,j_1},v_{\ell_1}]$ has length $m$ and set $P_1=C_{j_1}[t_{1,j_1},v_{\ell_1}]$.
Let $j_2$ be an index such that $C_{j_2}[t_{1,j_2},v_{\ell_k}]$ has length $m$ and set $P_2=C_{j_2}[t_{1,j_2},v_{\ell_k}]$.
Let $j_3$ be an index such that  $C_{j_3}[t_{1,j_3},v_{f}]$ has length $M_{\ell_{2k}}$ and set $P_3=C_{j_3}[v_f, s_{1,j_3}]$ (some vertices of $P_3$ are not in $I^+(C_1)$).

Note that any internal vertex $x$ of $P_1$ or $P_2$ has an integer in $\Dis(x)$
which is smaller than $m$ and every internal vertex $y$ of $P_3$ has an integer in $\Dis(y)$ which
is greater than $M_{\ell_{2k}}$, or does not belong to $I^+(C_1)$. Hence, 
$P_1$, $P_2$ and $P_3$ are disjoint from $P[v_{\ell_1},v_f]$.

We distinguish between the intersection of $P_1$, $P_2$ and $P_3$:

\begin{itemize}
	\item Suppose $P_3$ does not intersect $P_1 \cup P_2$.
	\begin{itemize}
		\item Assume first that $P_1$ and $P_2$ are disjoint. If $s(P_1)$ is in $C_1[t(P_3), s(P_2)]$, then the union of  $P_1 \odot P[v_{\ell_1}, v_{\ell_k}]$, $P[v_{\ell_k}, v_f] \odot P_3\odot C_1[t(P_3), s(P_1)]$ and $C_1[s(P_1), s(P_2)]\odot P_2$ is a subdivision of $B(k,1;k)$, a contradiction.
		 If $s(P_1)$ is in $C_1[ s(P_2), t(P_3)]$, then the union of  $C_1[s(P_2), s(P_1)]\odot P_1\odot P[v_{\ell_1}, v_{\ell_k}]$, $P[v_{\ell_k}, v_f]\odot  P_3\odot C_1[t(P_3), s(P_2)]$, and  $P_2$ is a subdivision of $B(k,1;k)$, a contradiction.
	
		\item Assume now $P_1$ and $P_2$ intersect. Let $u$ be the last vertex along $P_2$ on which they intersect. The union of $P_1[u,v_{\ell_1}]\odot P[v_{\ell_1}, v_{\ell_k}]$, $P[v_{\ell_k}, v_f]\odot P_3\odot C[t(P_3), s(P_1)]\odot P_1[s(P_1), u]$, and $P_2[u, v_{\ell_k}]$ is a subdivision of $B(k,1;k)$, a contradiction.
	\end{itemize}

	\item Assume $P_3$ intersect $P_1\cap P_2$. Let $v$ be the first vertex along $P_3$ in $P_1\cap P_2$ and let $u$ be the last vertex of $P_1\cap P_2$ along $P_2$. The union of $P_1[u,v_{\ell_1}]\odot P[v_{\ell_1}, v_{\ell_k}]$, $P[v_{\ell_k}, v_f]\odot P_3[v_f,v]\odot P_1[v, u]$, and $P_2[u, v_{\ell_k}]$ is a subdivision of $B(k,1;k)$, a contradiction.

	\item Assume now that $P_3$ intersects $P_1\cup P_2$ but not $P_1\cap P_2$. Let $v$ be the first vertex along $P_3$ in $P_1\cup P_2$.
	\begin{itemize}
	\item If $v \in P_2$, let $u$ be the last vertex on $P_2\cap P_3$ along $P_3$. Observe that $P_3[v,u]$ is also a subpath of $P_2$ and therefore contains no vertex of $P_1$. Furthermore, there is a dipath $Q$ from $u$ to $v_{\ell_1}$ in $P_3[u, t(P_3)]\cup C_1\cup  P_1$. Hence, the union of $P[v_{\ell_k}, v_f] \odot P_3[v_f,v]$, $Q\odot P[v_{\ell_1},v_{\ell_k}]$, and $P_2[u,v_{\ell_k}]$ is a subdivision of $B(k,1;k)$, a contradiction.
	
	\item If $v\in P_1$, let $u$ be the last vertex on $P_1\cap P_3$ along $P_3$.  Observe that $P_3[v,u]$ is also a subpath of $P_1$ and therefore contains no vertex of $P_2$. Furthermore, there is a dipath $Q$ from $u$ to $v_{\ell_k}$ in $P_3[u, t(P_3)]\cup C_1\cup  P_2$.
	The union of $P[v_{\ell_k}, v_f] \odot P_3[v_f,u]$, $P_1[u, v_{\ell_1}]\odot P[v_{\ell_1}, v_{\ell_k}]$ and $Q$ is a subdivision of $B(k,1;k)$, a contradiction.

	\end{itemize}
\end{itemize}
\end{subproof}

Claim~\ref{claim:c+} shows the existence of $\phi^+$ and completes the proof of Lemma~\ref{lem:IC}.
\end{proof}

\begin{lemma}\label{lem:col-union-cycle}
Let ${\cal C}$ be a $k$-suitable collection of directed cycles in a $B(k,1;k)$-\free\ digraph.
There exists a proper colouring $\phi$ of $\bigcup{\cal C}$ with $\alpha_k$ colours, such that, each subpath of length $7k$ of each cycle of $\mathcal{C}$ is rainbow-coloured.
\end{lemma}
\begin{proof}
We prove by induction on the number of cycles in $\mathcal{C}$ the following stronger statement:
{\it if there exists
a partial colouring $\phi$ such that one of the cycle $C_1$ has a path of length less than $7k$
which is rainbow-coloured, then we can extend this colouring to all $D[\mathcal{C}]$ using less
than $\alpha_k$ colours such that, on each cycle, every subpath of length
$7k$ is rainbow-coloured}.

Consider a rainbow-colouring of a subpath of length less than $7k$ of a cycle $C_1\in {\cal C}$.
By Lemma \ref{lem:IC}, we can extend this colouring to a colouring $\phi_1$ of $I(C_1)$ at most $\alpha_k$ colours.
Note that the non-coloured vertices of $\bigcup{\cal C}$ are in one of the connected components of $\bigcup{\cal C} - I(C_1)$.
Let $A$ be a connected component of $\bigcup{\cal C} - I(C_1)$. The coloured (by $\phi_1$) vertices of ${\cal C}\cap A$ are those of $({\cal C}\cap A)-A$. Hence, by Lemma \ref{lem:A}, they all belong to some cycle $C_j$ and so to the diptah $Q_j$ which has length at most $7k$.
Hence, by the induction hypothesis, we can extend $\phi_1$ to $A$. Doing this for each component, we extend $\phi_1$ to the whole $\bigcup{\cal C}$.
\end{proof}

Set $\beta_k= k(4k^2+2)(2\cdot (4k)^{4k} + 1) \alpha_k$.

\begin{lemma}\label{lem:DC}
Let ${\cal C}$ be a $k$-suitable collection of directed cycles in a $B(k,1;k)$-\free\ digraph $D$.
For every component $\mathcal{S}$ of $\mathcal{C}$, we have $\chi(D[\mathcal{S}]) \leq \beta_k$.
\end{lemma}

\begin{proof}
We define a sort of Breadth-First-Search for $\mathcal{S}$.
Let $C_0$ be a cycle of $\mathcal{S}$ and set $L_0 = \{C_0\}$. For every cycle $C_s$ of $\mathcal{S}\cap C_0$, we put $C_s$ in level $L_1$ and say that $C_0$ is the \textit{father} of $C_s$. We build  the levels $L_i$ inductively
until all cycles of $\mathcal{S}$ are put in a level : $L_{i+1}$ consists of every cycle $C_l$ not in $\bigcup_{j \leq i} L_j$ such that 
there exists a cycle in $L_i$ intersecting $C_l$.  For every $C_l\in L_{i+1}$, we choose one of the cycles $L_{i}$ intersecting it to be its {\it  father}. Henceforth every cycle in $L_{i+1}$ has a unique  father even though it might intersect many cycles of $L_i$.
A cycle $C$ is an {\it ancestor} of $C'$ if there is a sequence $C=C_1, \dots , C_q=C'$ such that $C_i$ is the father of $C_{i+1}$ for all $i\in [q-1]$.

For a vertex $x$ of $\bigcup \mathcal{S}$, we say that $x$ \textit{belongs to} level $L_i$ if
$i$ is the smallest integer such that there exists a cycle in $L_i$ containing $x$. Observe that the vertices of each cycle $C_l$ of ${\cal S}$ belong to consecutive levels, that is there exists $i$ such that $V(C_l)\subseteq L_i\cup L_{i+1}$.

To bound  the chromatic number of $D[{\cal S}]$, we partition its arc set of in $(A_0,A_1, A_2)$, where
\begin{itemize}
	\item $A_0$ is the set of arcs of $D[{\cal S}]$ which ends belong to the same level, and
	\item $A_1$ is the set of arcs of $D[{\cal S}]$ which ends belong to different levels $i$ and $j$ with $ | i - j| < k$.	
	\item $A_2$ is the set of arcs of $D[{\cal S}]$ which ends belong to different levels $i$ and $j$ with $ | i - j| \geq k$.
\end{itemize} 

For $i\in [3]$, let $D_i$ be the spanning subdigraph of $D[{\cal S}]$ with arc set $A_i$.
We shall now we bound the chromatic numbers of $D_0$, $D_1$ and  $D_2$.

\begin{claim}\label{claim:D1}
$\chi(D_1)\leq k$.
\end{claim}

\begin{subproof}
Let $\phi_1$ be the colouring that assigns to all vertices of level $L_i$ the colour $i$ modulo $k$, it 
is easy to see that $\phi_1$ is a proper colouring of $D_1$.
\end{subproof}

Let $C_l$ be a cycle of $L_i$, $i \geq 1$ and $C_{l'}$ its father. 
Let $p^+_l$ and $r^+_l$ be the vertices such that $C_l[t_{l,l'}, p^+_l]$ and $C_l[p^+_l,r^+_l]$  have length $k$. 
Let $p^-_l$ and $r^-_l$ be the vertices such that $C_l[p^-_l,s_{l,l'}]$ and $C_l[r^-_l,p^-_l]$ have length $k$.
Let $R^-_l$ be the set of vertices of $C_l]r^-_l, s_{l,l'}[$, $P^-_l$ the set of vertices of $C_l]p^-_l, s_{l,l'}[$,
$R^+_l$ the set of vertices of $C_l ]t_{l,l'}, r_l[$, $P^+_l$ the set of vertices of $C_l ]t_{l,l'}, r_l[$, and finally let 
$R'_l$ be the set of vertices belonging to $L_i$ in $C_l \setminus \{ R^+_l \cup R^-_l \}$.

\begin{claim}\label{claim:Rpos}
Let $x$ be a vertex in $L_i$ with $i\geq 1$. Let $C_l$ and $C_m$ be two cycles of $L_i$ containing $x$.
Then either $x \in P^+_l$ and  $x \in  P^+_m$, or $x \in P^-_l$ and  $x \in  P^-_m$.
\end{claim}

\begin{subproof}
Suppose for a contradiction that $x\in P^+_l$ and $x \not \in  P^+_m$  . 
Let $C_{l'}$ and $C_{m'}$ be the fathers of $C_l$ and $C_m$ respectively (they can be the same cycle). 
By definition of the $L_j$'s,  there exists a dipath $P$ from $t_{l,l'}$ to $s_{m,m'}$ only going through 
$C_{l'}$, $C_{s'}$ and their ancestors.
In particular $P$ is disjoint from $C_l - C_{l'}$ and  $C_s - C_{s'}$.
Observe that $C_l[s_{l,l'},t_{l,m}]$ has length at most $3k$ because it is contained in the union of $P_{l,l'}$, $P_{l,m}$, and $C_l[t_{l.l'}, x]$ which has length at most $k$ because $x\in P^+_l$. Hence  $C_l[t_{l,m},s_{l,l'}]$ has length at least $k$.
Moreover $C_m[s_{m,m'},t_{l,m}]$ contains  $C_m[t_{m,m'}, x]$ which has length at least $k$ because $x\notin P^+_m$.
Thus the union of $C_l[t_{l,m},s_{l,l'}] \odot P$, $C_m[t_{l,m},s_{m,m'}]$, 
and  $C_m[s_{m,m'},t_{l,m}]$ is a subdivision of $B(k,1;k)$, a contradiction.
The case where  $x\in P^-_l$ and $x \not \in  P^-_m$ is symmetrical and the case where $x$ does not belong to $P^-_l \cup P^+_l \cup P^-_m \cup P^+_m$ is identical.
\end{subproof}

Claim~\ref{claim:Rpos} imply that each level $L_i$ may be partitioned into sets $X^+_i$, $X^-_i$ and $X'_i$, where
$X^+_i$ (resp.  $X^-_i$) is the set of vertices $x$ of $L_i$ such that every $x\in R^+_l$ (resp. $x\in R^-_l$)
for every cycle $C_l$ of $L_i$ containing $x$ and $X'_i$ is set of vertices in $L_i$ but not in $X^+_i\cup X^-_i$. 
Set $X^+=V(C_0) \cup \bigcup_{\i\geq 1} X^+_i$, $X^-=\bigcup_{i\geq 1} X^-_i$ and $X'=\bigcup_{i\geq 1} X'_i$.
Clearly $(X^+, X^-, X')$ is a partition of $V(D[{\cal S}])$.


%

\begin{claim}\label{claim:D2}
$\chi(D_2)\leq 4k^2 +2 $.
\end{claim}

\begin{subproof}
Since $X^+\cup X^-\cup X-=V(D_2)$, we have $\chi(D_2) \leq \chi (D_2[X^+\cup X']) + \chi(D_2[X^+\cup X'])$.
We shall prove that $\chi (D_2[X^+\cup X'])\leq 2k^2+1$ and $\chi (D_2[X^-\cup X'])\leq 2k^2+1$ which imply the result.

\medskip

Let $x$ and $y$ be two adjacent vertices of $D_2[X^+\cup X']$. Let $L_i$ be the level of $x$ and $L_j$ be the level of $y$. Without loss of generality, we may assume that $j\geq i+k$.
Let $C_x$ be the cycle of $L_i$ such that $x \in C_x$ and $C_y$ the cycle of $L_j$ such that $y \in C_y$.
By considering ancestors of $C_x$ and $C_y$, there is a shortest sequence of cycles $C_1 \dots C_p$ such that $C_1 = C_x$ and $C_p = C_y$ and for all $l\in [p-1]$, either $C_l$ is the father of $C_{l+1}$ or $C_{l+1}$ is the father of $C_l$.
In particular $C_{p-1}$ is the father of $C_p$. Since $y\in X^+\cup X'$, then $C[y,t_{p-1,p}]$ has length at least $k$.

Assume that $xy$ is an arc. In $\bigcup_{l=1}^{p-1} C_l$, there is a dipath $P$ from $t_{p-1,p}$ to $x$. This path has length at least $k-1$ because
it must go through all levels $L_{i'}$, $i\leq i'\leq j-1$ because the vertices of any cycle of ${\cal S}$ are in two consecutive levels.
Hence the union of $P\odot (x,y)$, $C_p[t_{p-1,p}, y]$, and  $C_p[y,t_{p-1,p}]$ is a subdivision of $B(k,1;k)$, a contradiction.
Hence  $yx$ is an arc.

Suppose that $C_x$ is not an ancestor of $C_y$. In particular, $C_2$ is the father of $C_1$
and there exists a path $P$ from $t_{1,2}$ to $y$ in $\bigcup_{l=2}^{p-1} C_l$ of length at least $k-1$ and 
internally disjoint from $C_1$. Hence the union of $P \odot yx$, $C_1[x,t_{1,2}]$ and $C_1[t_{1,2},x]$ is a subdivsion of  $B(k,1;k)$.
Hence $C_x$ is an ancestor of $C_y$. 

In  particular, $C_l$ is the father of $C_{l+1}$ for all $l\in [p-1]$. Let $P$ be the dipath from $t_{1,2}$ 
to $y$ $\bigcup_{l=2}^{p} C_l$. It has length at least $k-1$ because it must go through all levels $L_i$, $1\leq i\leq p-1$. 
$C_1[x,t_{1,2}]$ has length less than $k$ , for otherwise the union of $P \odot yx$, $C_1[x,t_{1,2}]$ and $C_1[t_{1,2},x]$ would be a subdivision of  $B(k,1;k)$.

To summarize, the only arcs of $D_2[X^+\cup X']$ are arcs $yx$ such that $C_x$ is an ancestor of $C_y$ and $C_1[x,t_{1,2}]$ has length less than $k$ with $C_1 \dots C_p$ is the
sequence of cycles such that $C_1=C_x$ to $C_p=C_y$ and $C_l$ is the father of $C_{l+1}$ for all $l\in [p-1]$.
In particular, $D_2[X^+\cup X']$ is acyclic.

Let $y$ be a vertex of $D_2[X^+\cup X']$.
Let $L_p$ be the level of $y$ and let $C_0, \dots, C_p$ be the sequence of cycles such that $C_{l-1}$ is the father of $C_l$ for all $l\in [p]$.
For $0\leq l\leq p-1$, let $Q_l$ be the subdipath of $C_l$ of length $k-1$ terminating at $t_{l,l+1}$. By the above property, the out-neighbbours of
$y$ are in $\bigcup_{l=0}^{p-1} Q_l$.
Suppose for a contradiction that $y$ has out-degree at least $2k^2+1$. Then there are $2k+1$ distinct indices $l_1 < \cdot < l_{2k+1}$ such that for all $i\in [2k+1]$,  $C_{l_i}$ contains an out-neighbour $X_i$ of $y$.
Let $P$ be the shortest dipath from $x_1$ to $y$ in $\bigcup_{l=l_1}^p C_l$. This dipath intersect all cycles $C_l$ $l_1\leq l\leq p$. Let $z$ be first vertex of $P$ along $C_{l_{k+1}}[x_{k+1}, t_{l_{k+1}, l_{k+2}}]$. Vertex $z$ belongs to either $L_{l_{k+1}-1}$ or $L_{l_{k+1}}$.
Thus $P[x_1, z]$ and $P[z, y]$ have length at least $k-1$ and $k$ respectively since $P$ goes through all levels from $L_{l_1}$ to $L_p$. 
 Hence the union of $(y, x_1) \odot P[x_1, z]$,  $(y,x_{k+1}\odot C_{l_{k+1}}[x_{k+1},z]$, and
$P[z, y]$ is a subdivision of $B(k,1;k)$, a contradiction.
Therefore $D_2[X^+\cup X']$ has maximum out-degree at most $2k^2$. 

$D_2[X^+\cup X']$ is acyclic and has maximum out-degree at most $2k^2$. Therefore it is $2k^2$-degenerate, and so $ \chi(D_2[X^+\cup X'])\leq 2k^2 +1$.
By symmetry, we have $ \chi(D_2[X^-\cup X'])\leq 2k^2 +1$.
\end{subproof}

To bound $\chi(D_0)$ we partition the vertex set according to a colouring $\phi$ of $\bigcup {\cal S}$ given by Lemma~\ref{lem:col-union-cycle}.
For every colour $c\in [\alpha_k]$, let $X^+(c)$ be the set $X^+\cap \phi^{-1}(c)$ of vertices of $X^+$ coloured $c$, and $X^-(c)$ the set $X^-\cap \phi^{-1}(c)$ of vertices of $X^-$ coloured $c$. Similarly, let  $X^+_i(c)= X^+_i\cap \phi^{-1}(c)$ and $X^-_i(c)= X^-_i\cap \phi^{-1}(c)$.
We denote by $D^+_0(c)$ (resp. $D^-_0(c)$, $D'_0(c)$) the subdigraph of $D_0$ induced by the vertices of $X^+(c)$, (resp. $X^-(c)$, $X'(c)$).

\begin{claim} \label{claim:X'}
$\chi(D'_0(c))= 1$ for all $c\in [\alpha_k]$.
\end{claim}

\begin{subproof}
We need to prove that $D'_0(c)$ has no arc.
Suppose for a contradiction that $xy$ is an arc of $D'_0(c)$. 
By definition of $D_0$ $x$ and $y$ are in a same level $L_i$.
Let $C_l$ and $C_m$ be two cycles of $L_i$ such that $x \in C_l$ and $y \in C_m$.

If $C_l=C_m$, then both $C_l[x,y]$ and $C_l[y,x]$ have length at least $7k$ because the subdipaths of  length $7k$ of $C_l$ are rainbow-coloured by $\phi$. Hence the union of those paths and $(x,y)$ is  a subdivision of $B(k,1;k)$, a contradiction.
Henceforth, $C_l$ and $C_m$ are distinct cycles.

Suppose first that $C_l$ and $C_m$ intersect. By Claim~\ref{claim:Rpos}, $s_{l,m}$
belongs to $P^-_l$, $P^+_l$ or $L_{i-1}$, and by construction of $R'_l$, $C_l[x,s_{l,m}]$ and 
$C_l[s_{l,m},x]$ are both longer than $k$. Therefore they form with $(x,y) \odot C_m[y,s_{l,m}]$ a subdivision of $B(k,1;k)$, a contradiction.

Suppose now that $C_l$ and $C_m$ do not intersect. Let $C_l'$ and $C_m'$ be the fathers of $C_l$ and $C_m$ respectively. 
Let $P$ be the dipath  from $s_{m,m'}$ to $s_{l,l'}$ in of $\cup_{j<i} L_j$. Then the union of $C_l[s_{l,l'}, x]$,
$(x,y) \odot C_m[y, s_{m,m'}] \odot P$, and $C_l[x,s_{l,l'}]$ is a subdivision of $B(k,1,;k)$, a contradiction.
\end{subproof}

\begin{claim}\label{claim:D0}
$\chi(D^+_0(c))\leq (4k)^{4k}$ for all $c\in [\alpha_k]$. 
\end{claim}
\begin{subproof}
Set $p=(4k)^{4k}$.
Suppose for a contradiction that there exists $c$ such that $\chi(D^+_0(c)) > p$. 
Observe that $D^+_0(c)$ is the disjoint union of the $D[X^+_i(c)]$. Thus there exists a level $L_{i_0}$ such that $\chi(D[X^+_i(c)]) >p$.
Moreover $i_0>0$, because the vertices of $C_0$ coloured $c$ form a stable set.
 By Theorem~\ref{thm:gallairoy}, there exists a dipath $P=(v_0,  \dots , v_{p})$ of length $p$ in $D[X^+_i(c)]$.

Suppose that $P$ contains two vertices $x$ and $y$ of a same cycle $C$ of ${\cal S}$. 
Without loss of generality, we may assume that $P]x,y[$ contains no vertices of $C$.
Now  both $C[x,y]$ and $C[y,x]$ have length at least $7k$ because the subdipaths of  length $7k$ of $C$ are rainbow-coloured by $\phi$. 
Thus the union of $C[x,y]$, $P[x,y]$ and $C[y,x]$ is a subdivision of $B(k,1,;k)$, a contradiction.
Hence $P$ intersects every cycle of ${\cal S}$ at most once.

\medskip

For every $v\in V(P)$, let $\Len(v)$ be the set of lengths of $C_l[t_{l,l'},v]$ for all cycle $C_l \in  L_{i_0}$ containing $v$ and whose father is $C_{l'}$.

For $1\leq i\leq p$, let $m_i = \min \Len(v_i)$. By Claim~\ref{claim:Rpos}, $\Len(v_i) \subset [2k]$.
Lemma \ref{min} applied to $(m_i)_{1\leq i\leq p}$ yields a set $L$ of $4k^2$ indices of such that 
for any $i< j \in L$,  $m_i=m_j$ and $m_k > m_i$, for all $i< k < j$.
Let $l_1 < l_2 < \cdots < l_{4k^2}$ be the elements of $L$ and let $m= m_{l_1} = \cdots = m_{l_{4k^2}}$.

For $1\leq j\leq 4k^2-1$, let $M_j = \max \bigcup_{l_j\leq i < l_{j+1}} \Len(v_i)$.
By definition $M_j\leq 2k$.
Applying Lemma~\ref{max} to $(M_j)_{1\leq j\leq 4k^2}$,  we get a sequence of size $2k$ $M_{j_0+1} \dots M_{j_0+{2k}}$ such that $M_{j_0+{2k}}$ is the greatest. 
For sake of simplicity, we set $\ell_i =j_0+i$ for $1\leq i\leq 2k$.
Let $f$ be  the smallest index not smaller than $\ell_{2k}$ for which $M_{\ell_{2k}} \in \Len (v_f)$. 

Let $j_1$ and $j'_1$ be indices such that $v_{\ell_1} \in C_{j_1}$, $C_{j_1}$ is in $L_{i_0}$, $C_{j'_1}$ is the father of $C_{j_1}$ and 
$C_{j_1}[t_{j'_1,j_1},v_{\ell_1}]$ has length $m$. Set $P_1=C_{j_1}[t_{j'_1,j_1},v_{\ell_1}]$.
Let $j_2$ and $j'_2$ be indices such that $v_{\ell_k} \in C_{j_2}$, $C_{j_2}$ is in $L_{i_0}$, $C_{j'_2}$ is the father of $C_{j_2}$  and 
$C_{j_2}[t_{j'_2,j_2},v_{\ell_k}]$ has length $m$. Set $P_2=C_{j_2}[t_{j'_2,j_2},v_{\ell_k}]$.
Let $j_3$ and $j'_3$ be indices such that $v_{f} \in C_{j_3}$, $C_{j_3}$ is in $L_i$, $C_{j'_3}$ is the father of $C_{j_3}$  and 
$C_{j_3}[t_{j'_3,j_3}, v_{f}]$ has length $M_{\ell_{2k}}$. Set $P_3=C_{j_3}[v_{f},s_{j'_3,j_3}]$.
Note that any internal vertex $x$ of $P_1$ or $P_2$ has an integer in $\Len(x)$
which is smaller than $m$ and every internal vertex $y$ of $P_3$  either has an integer in $\Len(y)$ which
is greater than $M_{\ell_{2k}}$, or does not belong to $X^+(c)$. Hence, 
$P_1$, $P_2$ and $P_3$ are disjoint from $P[v_{\ell_1},v_f]$. 

We distinguishes cases according to the intersection between $P_1$, $P_2$ and $P_3$:
Let $P_5$ be a shortest dipath in $\cup_{i < i_0} L_i$ from $s_{j'_3,j_3}$ to $t_{j'_1,j_1}$ and
 $P_5$ be a shortest dipath in $\cup_{i < i_0} L_i$ from $s_{j'_3,j_3}$ to $t_{j'_2,j_2}$

\begin{itemize}
	\item Suppose $P_3$ does not intersect $P_1 \cup P_2$. 
	\begin{itemize}
		\item Suppose $P_1$ and $P_2$ are disjoint and let $P_4$ be the shortest dipath in $\cup_{i < i_0} L_i$ 
from $t_{j'_1,j_1}$ to $t_{j'_2,j_2}$. Let $v$ be the last vertex of $P_4$ in $P_4 \cap P_5$.
The union of $P_5[v, t_{j'_1,j_1}] \odot P_1 \odot P[v_{\ell_1}, v_{\ell_k}]$, $P_4[v,t_{j'_2,j_2}] \odot P_2$,
and $P[v_{\ell_k}, v_f] \odot P_3 \odot P_5[s_{j'_3,j_3}, v]$ is a subdivision of $B(k,1;k)$,  a contradiction.

		\item Assume now $P_1$ and $P_2$ intersect. Let $u$ be the last vertex along $P_2$ on which they intersect.
		 The union of $P_1[u,v_{\ell_1}]\odot P[v_{\ell_1}, v_{\ell_k}]$, $P_2[u, v_{\ell_k}]$, and $P[v_{\ell_k}, v_f]\odot P_3\odot P_5 \odot P_1[t_{j'_1,j_1}, u]$ is a subdivision of $B(k,1;k)$, a contradiction.
	\end{itemize}

	\item Assume $P_3$ intersects $P_1\cap P_2$. Let $v$ be the first vertex along $P_3$ in $P_1\cap P_2$ and let $u$ be the last vertex of $P_1\cap P_2$ along $P_2$. The union of $P_1[u,v_{\ell_1}]\odot P[v_{\ell_1}, v_{\ell_k}]$, $P_2[u, v_{\ell_k}]$, and $P[v_{\ell_k}, v_f]\odot P_3[v_f,v]\odot P_1[v, u]$ is a subdivision of $B(k,1;k)$, a contradiction.

	\item Assume now that $P_3$ intersect $P_1\cup P_2$ but not $P_1\cap P_2$. Let $v$ be the first vertex along $P_3$ in $P_1\cup P_2$.
	\begin{itemize}
	\item If $v \in P_2$, let $u$ be the last vertex of $P_2\cap P_3$ along $P_3$. Observe that $P_3[v,u]$ is also a subpath of $P_2$ and therefore contains no vertex of $P_1$. Hence, the union of $P_3[u, s_{j'_3,j_3}]\odot P_5 \odot P_1\odot P[v_{\ell_1},v_{\ell_k}]$,  $P_2[u,v_{\ell_k}]$, and $P[v_{\ell_k}, v_f] \odot P_3[v_f,v]$ is a subdivision of $B(k,1;k)$, a contradiction.

	\item If $v\in P_1$, let $u$ be the last vertex of $P_1\cap P_3$ along $P_3$.  Observe that $P_3[v,u]$ is also a subpath of $P_1$ and therefore contains no vertex of $P_2$. 
	Hence the union of $P_1[u, v_{\ell_1}]\odot P[v_{\ell_1}, v_{\ell_k}]$, $P_3[u, s_{j'_3,j_3}]\odot P_6 \odot P_2$, and $P[v_{\ell_k}, v_f] \odot P_3[v_f,u]$,  is a subdivision of $B(k,1;k)$, a contradiction.

	\end{itemize}

\end{itemize}
\end{subproof}

Similarly to Claim~\ref{claim:D0}, one proves that $\chi(D^-_0(c))\leq (4k)^{4k}$ for all $c\in [\alpha_k]$.
Hence, $\chi(D_0(c) \leq \chi(D^+_0(c)) + \chi(D^-_0(c)) + \chi(D'_0(c) \leq 2\cdot (4k)^{4k} + 1$.
Thus  $$\chi(D_0) \leq (2\cdot (4k)^{4k} + 1) \alpha_k.$$

Via Lemma~\ref{lem:decomp}, this equation and Claims~\ref{claim:D1} and  \ref{claim:D2} yields 
$$\chi(D) \leq \chi(D_0)\times \chi(D_1) \times \chi(D_2) \leq k(4k^2+2)(2\cdot (4k)^{4k} + 1) \alpha_k = \beta_k.$$
\end{proof}

\subsection{Proof of Theorem \ref{th:main}}

Consider $\mathcal{C}$ be a maximal $k$-suitable collection of cycles in $D$. Let $D'$ be the digraph obtained by contracting 
every strong component $S$ of $\bigcup \mathcal{C}$ (which is $\bigcup {\cal S}$ for some component ${\cal S}$ of ${\cal C}$) into one vertex. For each connected component ${\cal S}_i$ we call $s_i$ the
new vertex created.

\begin{claim}\label{cl:nocycle}
$\chi(D') \leq 8k$.
\end{claim} 

\begin{proof}
First note that since $D$ is strong so is $D'$. 

Suppose for a contradiction that $\chi(D') > 8k$. By Theorem \ref{thm:bondy}, there exists
a directed cycle $C = (x_1, x_2, \dots , x_l, x_1)$ of length at least $8k$. 	
For each vertex $x_j$ that corresponds to a $s_i$ in $D$, the arc $x_{j-1}x_j$ corresponds in $D$ to an arc whose head is a vertex $p_i$ of $S_i$ and the arc $x_jx_{j+1}$ corresponds to an arc whose tail is a vertex $l_i$ of $S_i$. Let $P_j$ be the dipath 
		from $p_i$ to $l_i$ in $\bigcup \mathcal{C}$. Note that this path intersects the elements of $S_i$ only along a subdipath. 
Let $C'$ be the cycle obtained from $C$ where we replace all contracted vertices $x_j$ by the path $P_j$. 
First note that $C'$ has length at least $8k$. Moreover, a cycle of $\mathcal{C}$  can
intersect $C'$ only along one $P_j$, because they all correspond to different strong components of $\bigcup \mathcal{C}$. Thus $C'$ intersects  each cycle of $\mathcal{C}$ on a subdipath. Moreover this subdipath has length smaller than $k$ for otherwise $D$ would contain 
a subdivision of $B(k,1;k)$. So $C'$ is a directed cycle of length at least $8k$ which intersects every cycle of $\mathcal{C}$ along a 
subdipath of length less than $k$. This contradicts the maximality of $\mathcal{C}$.
\end{proof}

Using Lemma \ref{lem:contrac} with
Claim~\ref{cl:nocycle} and Lemma~\ref{lem:DC}, we get that $\chi(D) \leq 8k \cdot \beta_k$.
This proves Theorem~\ref{th:main} for $\gamma_k=8k\cdot \beta_k = 8k^2(4k^2+2)(2\cdot (4k)^{4k} + 1) (2\cdot \col + 14k)$.


\begin{thebibliography}{XX}

\bibitem{AHT07}
L.~Addario-Berry, F.~Havet, and S.~Thomass{\'e}.
\newblock Paths with two blocks in $n$-chromatic digraphs.
\newblock \emph{Journal of Combinatorial Theory, Series B}, 97
  (4): 620--626, 2007.

\bibitem{AHS+13}
L.~Addario-Berry, F.~Havet, C.~L. Sales, B.~A. Reed, and S.~Thomass{\'e}.
\newblock Oriented trees in digraphs.
\newblock \emph{Discrete Mathematics}, 313 (8): 967--974,
  2013.

\bibitem{AC+16}
P. Aboulker, N. Cohen, F. Havet, W. Lochet, P. Moura and S. Thomassé
\newblock Subdivisions in digraphs of large out-degree or large dichromatic number
\newblock arXiv:1610.00876


   \bibitem{Bon76} J. A. Bondy, Disconnected orientations and a conjecture of Las Vergnas,
{\it J. London Math. Soc. (2)}, {\bf 14} (2) (1976), 277--282.


  \bibitem{BoMu08}
J.A. Bondy and U.S.R. Murty.
\newblock {\em {G}raph {T}heory}, volume 244 of {\em Graduate Texts in
  Mathematics}.
\newblock Springer, 2008.



\bibitem{Burr80}
S.~A. Burr.
\newblock Subtrees of directed graphs and hypergraphs.
\newblock In \emph{Proceedings of the 11th Southeastern Conference on
  Combinatorics, Graph theory and Computing}, pages 227--239, Boca Raton - FL,
  1980. Florida Atlantic University.

\bibitem{Bur82} S. A. Burr,
Antidirected subtrees of directed graphs.
{\it Canad. Math. Bull.} {\bf 25} (1982), no. 1, 119--120.

\bibitem{CHLN16}
N. Cohen, F. Havet, W. Lochet, and N. Nisse.
\newblock Subdivisions of oriented cycles in digraphs with large chromatic number.
\newblock arXiv:1605.07762

\bibitem{Erd59}
P.~Erd{\H{o}}s.
\newblock Graph theory and probability.
\newblock {\em Canad. J. Math.}, 11:34--38, 1959.

\bibitem{ErHa66}
P.~Erd{\H{o}}s and A. Hajnal.
\newblock On chromatic number of graphs and set-systems.
\newblock {\em Acta Mathematica Academiae Scientiarum Hungarica}, 17(1-2):61--99, 1966.



\bibitem{Gal68}
T.~Gallai.
\newblock On directed paths and circuits.
\newblock In \emph{Theory of Graphs (Proc. Colloq. Titany, 1966)}, pages
  115--118. Academic Press, New York, 1968.

\bibitem{Gya92}
A. Gy\'arf\'as.
\newblock Graphs with $k$ odd cycle lengths.
\newblock {\it Discrete Math.}, 103, pp. 41--48, 1992.


\bibitem{Has64}
M.~Hasse.
\newblock Zur algebraischen bergr\"und der graphentheorie {I}.
\newblock \emph{Math. Nachr.}, 28: 275--290, 1964.


\bibitem{KKPM}
R. Kim, SJ. Kim, J. Ma; B. Park
\newblock Cycles with two blocks in $k$-chromatic digraphs
\newblock arXiv:1610.05839



\bibitem{Roy67}
B.~Roy.
\newblock Nombre chromatique et plus longs chemins d'un graphe.
\newblock \emph{Rev. Francaise Informat. Recherche Op\'erationnelle},
  1 (5): 129--132, 1967.


\bibitem{Sum81}
D.~P. Sumner.
\newblock Subtrees of a graph and the chromatic number.
\newblock In {\em The theory and applications of graphs (Kalamazoo, Mich.,
  1980)}, pages 557--576. Wiley, New York, 1981.

\bibitem{Vit62}
L.~M. Vitaver.
\newblock Determination of minimal coloring of vertices of a graph by means of
  boolean powers of the incidence matrix.
\newblock \emph{Doklady Akademii Nauk SSSR}, 147: 758--759, 1962.




\end{thebibliography}
\end{document}